\documentclass{amsart}
\usepackage{graphics}
\usepackage{amsfonts,amsmath}

\begin{document}

%
%
%
 \newtheorem{thm}{Theorem}[section]
 \newtheorem{cor}[thm]{Corollary}
 \newtheorem{lem}[thm]{Lemma}
 \newtheorem{prop}[thm]{Proposition}
 \newtheorem{defn}[thm]{Definition}
 \newtheorem{assumption}[thm]{Assumption}
 \newtheorem{rem}[thm]{Remark}
 \newtheorem{ex}{Example}
\numberwithin{equation}{section}
\def\e{{\rm e}}
\def\bx{\mathbf{x}}
\def\bsmalpha{{\small \mathbf{\alpha}}}
\def\btheta{\mathbf{\theta}}
\def\balpha{\mathbf{\alpha}}
\def\by{\mathbf{y}}
\def\bz{\mathbf{z}}
\def\F{\mathcal{F}}
\def\R{\mathbb{R}}
\def\T{\mathbf{T}}
\def\N{\mathbb{N}}
\def\K{\mathbf{K}}
\def\Q{\mathbf{Q}}
\def\M{\mathbf{M}}
\def\O{\mathbf{O}}
\def\C{\mathbb{C}}
\def\P{\mathbf{P}}
\def\Z{\mathbb{Z}}
\def\H{\mathcal{H}}
\def\A{\mathbf{A}}
\def\V{\mathbf{V}}
\def\AA{\overline{\mathbf{A}}}
\def\B{\mathbf{B}}
\def\c{\mathbf{C}}
\def\L{\mathbf{L}}
\def\bS{\mathbf{S}}
\def\H{\mathcal{H}}
\def\I{\mathbf{I}}
\def\Y{\mathbf{Y}}
\def\X{\mathbf{X}}
\def\f{\mathbf{f}}
\def\z{\mathbf{z}}
\def\d{\hat{d}}
\def\bx{\mathbf{x}}
\def\y{\mathbf{y}}
\def\w{\mathbf{w}}
\def\b{\mathbf{b}}
\def\a{\mathbf{a}}
\def\u{\mathbf{u}}
\def\s{\mathcal{S}}
\def\cc{\mathcal{C}}
\def\co{{\rm co}\,}
\def\vol{{\rm vol}\,}
\def\om{\mathbf{\Omega}}

\title{On representations of the feasible set in convex optimization}
\author{Jean B. Lasserre}
\address{LAAS-CNRS and Institute of Mathematics\\
University of Toulouse\\
LAAS, 7 avenue du Colonel Roche\\
31077 Toulouse C\'edex 4,France}
\email{lasserre@laas.fr}
\date{}

\begin{abstract}
We consider the convex optimization problem 
$\min_{\bx} \{f(\bx): g_j(\bx)\leq 0,\,j=1,\ldots,m\}$
where $f$ is convex, the feasible set $\K$ is convex and Slater's condition holds,
but the functions $g_j$'s are {\it not} necessarily convex. 
We show that for any representation of $\K$ that 
satisfies a mild nondegeneracy assumption,
every minimizer is a Karush-Kuhn-Tucker (KKT) point and conversely
every KKT point is a minimizer. That is, the KKT optimality conditions are necessary and sufficient
as in convex programming where one assumes that the $g_j$'s are convex.
So in convex optimization, and as far as one is concerned with KKT points,
what really matters is the geometry of $\K$ and not so much its representation.
\end{abstract}

\keywords{Convex programming; KKT conditions}

\subjclass{90C25 65K05}

\maketitle
\section{Introduction}~
Given differentiable functions $f,g_j :\R^n\to\R$, $j=1,\ldots,m$,
consider the following convex optimization problem:
\begin{equation}
\label{defpb}
 f^*\,:=\,\displaystyle\inf_{\bx}\:\{\:f(\bx)\::\:\bx\in\K\:\}
\end{equation}
where $f$ is convex and the feasible set $\K\subset\R^n$ is convex and represented in the form:
\begin{equation}
\label{setk}
\K\,=\,\{\:\bx\in\R^n\::\: g_j(\bx)\,\leq\,0,\:j=1,\ldots,m\:\}.
\end{equation}
Convex optimization usually refers to minimizing a convex function over a convex set 
without precising its representation (see e.g. Ben-Tal and Nemirovsky \cite[Definition 5.1.1]{bental}
or Bertsekas et al. \cite[Chapter 2]{bertsekas}), and
it is well-known that convexity of the function $f$ and of the set $\K$ imply that every local minimum is a global minimum. An elementary proof
only uses the geometry of $\K$, not its representation by the defining functions $g_j$; see e.g. Bertsekas et al. \cite[Prop. 2.1.2]{bertsekas}. 

The convex set $\K$ may be represented by different choices of 
the (not necessarily convex) defining functions $g_j$, $j=1,\ldots,m$.
For instance, the set
\[\K\,:=\,\{\bx\in\R^2\,:\,1-x_1x_2\leq 0; \:\bx\geq 0\}\]
is convex but the function $\bx\mapsto 1-x_1x_2$ is not convex on $\R^2_+$.
Of course, depending on the choice of the defining functions $(g_j)$, several properties may or may not hold. In particular, the celebrated Karush-Kuhn-Tucker (KKT)
optimality conditions {\it depend} on the representation of $\K$.
Recall that $\bx\in\K$ is a KKT point if 
\begin{equation}
\label{kkt}
\nabla f(\bx)+\sum_{j=1}^m\lambda_j\nabla g_j(\bx)=0\quad\mbox{and}\quad  \lambda_j\,g_j(\bx)=0,\:j=1,\ldots,m,\end{equation}
for some nonnegative vector $\lambda\in\R^m$. (More precisely $(\bx,\lambda)$ is a KKT point.)

{\it Convex programming} refers to
the situation where $f$ is convex {\em and} the defining functions $g_j$ of $\K$
are also convex. See for instance Ben-Tal and Nemirovsky \cite[p. 335]{bental},
Berkovitz \cite[p. 179]{berkovitz},
Boyd and Vandenberghe \cite[p. 7]{boyd}, Bertsekas et al. \cite[\S 3.5.5]{bertsekas},
Nesterov and Nemirovskii \cite[p. 217-218]{nesterov}, and Hiriart-Urruty \cite{jbhu}.

A crucial feature of convex programming
is that when Slater's condition holds\footnote{Slater's condition 
holds for $\K$ if for some $\bx_0\in\K$, $g_j(\bx_0)<0$ for every $j=1,\ldots,m$.},
the KKT optimality conditions (\ref{kkt}) are necessary and sufficient, which shows that 
a representation of the convex set $\K$ with convex functions $(g_j)$ 
has some very attractive features.

The purpose of this note is to show that in fact, when $\K$ is convex and as far as one is concerned
with KKT points, what really matters is the 
geometry of $\K$ and not so much its representation.
Indeed, we show that if $\K$ is convex and Slater's condition holds 
then the KKT optimality conditions (\ref{kkt}) are also necessary and sufficient
for {\it all} representations of $\K$ that satisfy a mild nondegeneracy condition,
no matter if the $g_j$'s are convex. So this attractive feature is not specific to representations of $\K$ with convex functions.

That a KKT point is a local (hence global) minimizer follows easily from 
the convexity of $\K$.
More delicate is the fact that any local (hence global) minimizer is a KKT point.
Various {\it constraint qualifications} are usually required to hold at a minimizer,
and when the $g_j$'s
are convex the simple Slater's condition is enough. Here we show that 
Slater's condition is also sufficient for all representations of $\K$
that satisfy a mild additional nondegeneracy assumption on the boundary 
of $\K$. Moreover under Slater's condition this mild 
nondegeneracy assumption is automatically satisfied if the $g_j$'s are convex.
 
\section{Main result}

Let $\K\subset\R^n$ be as in (\ref{setk}). We first start with the following non degeneracy assumption:
\begin{assumption}[nondegeneracy]
\label{ass1}
For every $j=1,\ldots,m$, 
\begin{equation}
\label{ass-1}
\nabla g_j(\bx)\neq 0,\qquad \mbox{whenever }\bx\in\K\quad\mbox{and}\quad g_j(\bx)=0.
\end{equation}
\end{assumption}
Observe that under Slater's condition, (\ref{ass-1}) 
is automatically satisfied
if $g_j$ is convex. Indeed if $g_j(\bx)=0$ and $\nabla g_j(\bx)=0$ then by convexity
$0$ is the global minimum of $g_j$ on $\R^n$. Hence
there is no $\bx_0\in\K$ with $g_j(\bx_0)<0$.
We next state the following characterization of convexity.
\begin{lem}
\label{lemma1}
With $\K\subset\R^n$ as in (\ref{setk}), let Assumption \ref{ass1} and Slater's condition both hold for $\K$. Then $\K$ is convex if and only if for every $j=1,\ldots,m$:
\begin{equation}
\label{lemma-1}
\langle \nabla g_j(\bx),\y-\bx\rangle \,\leq\,0,\quad\forall\, \bx,\y\in\K\quad\mbox{with}\quad g_j(\bx)=0.\end{equation}
\end{lem}
\begin{proof}
{\em Only if} part.  Assume that $\K$ is convex and 
$\langle \nabla g_j(\bx),\y-\bx\rangle >0$ 
for some $j\in\{1,\ldots,m\}$ and some $\bx,\y\in\K$ with $g_j(\bx)=0$.
Then $g_j(\bx+t(\y-\bx))>0$ for all sufficiently small $t$, in contradiction with
$\bx+t(\y-\bx)\in\K$ for all $0\leq t\leq 1$
(by convexity of $\K$).

{\em If} part.  By (\ref{lemma-1}), at every point $\bx$ on the boundary of $\K$,
there exists a supporting hyperplane for $\K$. As $\K$ is closed with nonempty interior, by \cite{schneider}[Th. 1.3.3] the set $\K$ is convex\footnote{The author wishes to thank Prof. L. Tuncel 
for providing him with the reference \cite{schneider}.}.
\end{proof}

\begin{thm}
\label{main}
Consider the nonlinear programming problem (\ref{defpb}) and let Assumption
\ref{ass1}  and Slater's condition both hold. If $f$ is convex then
every minimizer is a KKT point and conversely, every KKT point is a minimizer.
\end{thm}
\begin{proof}
Let $\bx^*\in\K$ be a minimizer (hence a global minimizer) with $f^*=f(\bx^*)$. We first prove that $\bx^*$ is a KKT point. The Fritz-John optimality conditions state that
\[\lambda_0\nabla f(\bx^*)+\sum_{j=1}^m\lambda_j\,\nabla g_j(\bx^*)\,=\,0;\quad \lambda_jg_j(\bx^*)=0,\:j=1,\ldots,m,\]
for some non trivial nonnegative vector $0\neq\lambda\in\R^{m+1}$. See e.g. Hiriart-Urruty \cite[Th. page 77]{jbhu} or Polyak \cite[Theor. 1, p. 271]{polyak}. We next prove that $\lambda_0\neq0$. Suppose that $\lambda_0=0$ and let
$J:=\{j\in\{1,\ldots,m\}:\lambda_j >0\}$. As $\lambda\neq0$ and $\lambda_0=0$, the set $J$ is nonempty. Next, as $g_j(\bx_0)<0$ for every $j=1,\ldots,m$, there is some $\rho >0$ such that $B(\bx_0,\rho):=\{\z\in\R^n:\Vert \z-\bx_0\Vert<\rho\}\subset\K$ and $g_j(\z)<0$ for all $\z\in B(\bx_0,\rho)$ and all $j\in J$.
Therefore we obtain
\[\sum_{j\in J}\lambda_j\,\langle \nabla g_j(\bx^*),\z-\bx^*\rangle\,=\,0\,\quad\forall \,\z\in B(\bx_0,\rho),\]
which, by Lemma \ref{lemma1}, implies that $\langle \nabla g_j(\bx^*),\z-\bx^*\rangle=0$ 
for every $j\in J$ and every $\z\in B(\bx_0,\rho)$. But this clearly implies that $\nabla g_j(\bx^*)=0$ for every $j\in J$, in contradiction with Assumption \ref{ass1}.
Hence $\lambda_0>0$ and we may and will set $\lambda_0=1$, so that the KKT conditions hold at 
$\bx^*$. 

Conversely, let $\bx\in\K$ be an arbitrary KKT point, i.e., $\bx\in\K$ satisfies
\[\nabla f(\bx)+\sum_{j=1}^m\lambda_j\,\nabla g_j(\bx)\,=\,0;\quad \lambda_jg_j(\bx)=0,\:j=1,\ldots,m,\]
for some nonnegative vector $\lambda\in\R^{m}$. Suppose that there exists 
$\y\in\K$ with $f(\y)<f(\bx)$. Then we obtain the contradiction:
\begin{eqnarray*}
0&>&f(\y)-f(\bx)\\
&\geq&\langle \nabla f(\bx),\y-\bx\rangle
\quad\mbox{[by convexity of $f$]}\\
&=&-\sum_{j=1}^m\lambda_j \langle
\nabla g_j(\bx),\y-\bx\rangle\geq 0
\end{eqnarray*}
where the last inequality follows from $\lambda\geq0$ and Lemma \ref{lemma1}. 
Hence $\bx$ is a minimizer.
\end{proof}
Hence  if $\K$ is convex and both Assumption \ref{ass1} and Slater's condition hold,
there is a one-to-one correspondence between KKT points and minimizers. That is,
the KKT optimality conditions are necessary and sufficient for all representations of $\K$ that
satisfy Slater's condition and Assumption \ref{ass1}.

However there is an important additional property when all the defining functions
$g_j$ are convex. 
Dual methods of the type
\[\sup_{\lambda\in\R^m_+}\,\left\{\inf_{\bx} f(\bx)+\sum_{j=1}^m\lambda_jg_j(\bx)\,\right\},\]
are well defined because $\bx\mapsto f(\bx)+\sum_{j=1}^m\lambda_jg_j(\bx)$ is a convex function. In particular,
the Lagrangian $\bx\mapsto L_f(\bx):=f(\bx)-f^*+\sum_{j=1}^m \lambda_j g_j(\bx)$,
defined from an arbitrary KKT point $(\bx^*,\lambda)\in\K\times\R^m_+$, is convex and nonnegative on $\R^n$, with $\bx^*$ being a global minimizer. 
If the $g_j$'s are not convex this is not true in general.

\begin{ex}
{\rm 
Let $n=2$ and consider the problem 
\[\P: \quad f^*=\min\,\{\,f(\bx)\::\:a-x_1x_2\leq0;\: \A\bx\leq \b;\,\bx\geq 0\:\},\]
where $a>0$, $\A\in\R^{m\times n}$, $\b\in\R^m$, and $f$ is convex and differentiable.
The set 
\[\K\,:=\,\{\bx\in\R^2\::\: a -x_1x_2\leq 0;\:\A\bx\leq\b;\:\bx\geq 0\:\}\]
is convex and it is straightforward to check that Assumption \ref{ass1} holds.
Therefore, by Theorem \ref{main}, if Slater's condition holds, every KKT point is a global minimizer. However, the Lagrangian
\[\bx\mapsto f(\bx)-f^*+\psi (a-x_1x_2)+\langle\lambda,\A\bx-\b\rangle-
\langle\mu,\bx\rangle,\]
with nonnegative $(\psi,\lambda,\mu)\in\R\times\R^m\times\R^n$, may not be convex whenever 
$\psi\neq0$ (for instance if $f$ is linear). On the other hand, notice that $\K$ has the equivalent
convex representation
\[\K\,:=\,\left\{\bx\in\R^2\::\: \left[\begin{array}{cc}x_1 & \sqrt{a}\\
\sqrt{a}&x_2\end{array}\right]\succeq0;\:\A\bx\leq\b\:\right\},\]
where for a real symmetric matrix $\B$, the notation
$\B\succeq0$ stands for $\B$ is positive semidefinite.
}
\end{ex}
A topic of further investigation is concerned with computational efficiency.
Can {\it efficient} algorithms be devised for some class of convex problems (\ref{defpb}) where
the defining functions $g_j$ of $\K$ are not necessarily convex?

\end{document}